\documentclass[12pt]{amsart}

\usepackage[latin1]{inputenc}
\usepackage[english]{babel}
\usepackage{amsthm}
\usepackage{amsmath}
\usepackage{amsxtra}
\usepackage{mathrsfs}
\usepackage{amssymb}
\usepackage[dvips]{graphicx}
\usepackage{graphicx}

\newcommand{\urltilde}{\kern -.15em\lower .7ex\hbox{~}\kern .04em}

\newtheorem{theorem}{Theorem}[section]

\newtheorem{lemma}[theorem]{Lemma}
\newtheorem{fact}[theorem]{Fact}
\newtheorem*{claim*}{Claim}

\newtheorem*{theorem*}{Theorem}
\newtheorem{claim}{Claim}
\newtheorem*{subclaim*}{Subclaim}

\newtheorem{cor}[theorem]{Corollary}

\newcommand{\mrm}{\mathrm}
\newcommand{\mbb}{\mathbb}
\newcommand{\mcal}{\mathcal}

\newcommand{\rst}{\!\upharpoonright\!}
\newcommand{\sref}{\mathrm{SR}}
\newcommand{\ssref}{\mathrm{SSR}}

\newcommand{\MAC}{\rm MA_{\omega_1}(Cohen)}
\newcommand{\MM}{{\rm MM}}
\newcommand{\PFA}{{\rm PFA}}
\newcommand{\SPFA}{\rm SPFA}

\newcommand{\SR}{\rm SR}
\newcommand{\SSR}{\rm SSR}

\newcommand{\TP}{\rm TP}
\newcommand{\ITP}{\rm ITP}
\newcommand{\Fn}{\rm Fn}
\newcommand{\lev}{\rm lev}
\newcommand{\F}{\mathscr{F}}
\newcommand{\G}{\mathscr{G}}
\newcommand{\B}{\mathscr{B}}
\newcommand{\Dcal}{\mathscr{D}}
\newcommand{\Ecal}{\mathscr{E}}
\newcommand\Pcal{\mathscr{P}}
\newcommand\Qcal{\mathscr{Q}}

\newcommand{\defarrow}{\stackrel{\mathrm{def}}{\Leftrightarrow}}

\newcommand{\otype}{\mathop{\mathrm{otp}} \nolimits}
\newcommand{\cof}{\mathop{\mathrm{cof}} \nolimits}
\newcommand{\ltpt}{\mathop{\mathrm{lim}} \nolimits}
\newcommand{\bsup}{{\sup } ^+}
\newcommand{\skull}{\mathop{\mathrm{Hull}} \nolimits}
\newcommand{\clsr}{\mathop{\mathrm{cl}} \nolimits}

\newcommand{\dom}{\mathop{\mathrm{dom}} \nolimits}

\newenvironment{renumerate}%
{\begin{enumerate}}{\end{enumerate}}
{\begin{enumerate}}{\end{enumerate}}
\newenvironment{aenumerate}%
{\begin{enumerate}}{\end{enumerate}}

\begin{document}

\title[Stationary reflection principles]{Stationary reflection principles and two cardinal
tree properties}
\author{Hiroshi Sakai}
\email{hsakai@people.kobe-u.ac.jp}
\address{Department of Computer Science and Systems Engineering
Kobe University, 1-1 Rokkodai, Nada, Kobe 657-8501, Japan}
\urladdr{http://kurt.scitec.kobe-u.ac.jp/\urltilde hsakai}
\author[Boban Veli\v{c}kovi\'{c}]{Boban Veli\v{c}kovi\'{c}}
\email{boban@math.univ-paris-diderot.fr}
\address{Institut de Mathematiques de Jussieu - Paris Rive Gauche,
Universit\'e Paris Diderot,
75205 Paris Cedex 13,
France}
\urladdr{http://www.logique.jussieu.fr/\urltilde boban}
\thanks{This work is supported by the Exchange Grant 2744 from the INFTY Research Networking Programme
(E.S.F.) and by a binational grant PHC Sakura 27604XE.
The first author was supported by Grant-in-Aid for Young Scientists (B) No. 23740076 of the Ministry of Education,
Culture, Sports, Science and Technology Japan (MEXT)}

\keywords{stationary sets, semi stationary sets, reflection, singular cardinal hypothesis, tree property, supercompact
cardinals, strongly compact cardinals}

\subjclass[2000]{Primary: 03E04, 03E35, 03E50, 03E55; Secondary:
03E05, 03E65}

\begin{abstract}{We study consequences of stationary and semi-stationary set reflection.
We show that the semi-stationary reflection principle implies the Singular Cardinal Hypothesis,
the failure of weak square principle, etc. We also consider two cardinal tree properties
introduced recently by Weiss and prove that they follow from stationary and semi-stationary
set reflection augmented with a weak form of Martin's Axiom. We also show that there
are some differences between the two reflection principles which suggest that stationary
set reflection is analogous to supercompactness whereas semi-stationary set reflection
is analogous to strong compactness. }
\end{abstract}
\maketitle


\section*{Introduction}

Reflection principles are a way of transferring large cardinal properties to small cardinals.
Over the years a large number of such principles have been considered and a rich theory
has been developed with numerous applications not only to pure set theory but also to various
other areas of mathematics. Some of the earliest and most important
reflection principles concern reflection of various classes of stationary sets.
In this paper we will consider stationary reflection principle $\sref$, introduced by Foreman,
Magidor and Shelah \cite{FMS}, which asserts that, for every $\lambda \geq \omega_2$,
the following statement $\sref (\lambda)$ holds:
\medskip
\begin{quote}{\em
If $S$ is a stationary subset of $[\lambda]^{\omega}$ then there is
$I\subseteq \lambda$ of cardinality $\omega_1$ such that $\omega_1\subseteq I$
and  $S\cap [I]^{\omega}$ is stationary in $[I]^{\omega}$.}
\end{quote}
\medskip
\noindent $\sref$ and its variations have been studied extensively by a number of authors
and it has been shown that it has important consequences in cardinal arithmetic,
infinite combinatorics, topology, algebra, etc. One of the key observations of
\cite{FMS} is that $\sref$ implies the following principle $( \dagger )$:

\medskip
\begin{quote}{\em
Every posets preserving stationary subsets of $\omega_1$ is semiproper.}
\end{quote}
\medskip
\noindent
This allowed Foreman, Magidor and Shelah \cite{FMS} to show that in the standard
model for the Semi Proper Forcing Axiom ($\SPFA$) a provably maximal forcing axiom,
Martin's Maximum ($\MM$), holds. Somewhat later Shelah \cite{Shelah_SPFA} showed
that $\MM$ follows outright from $\SPFA$. The principle $( \dagger )$ in itself has many
important consequences, for instance, already in \cite{FMS} it was shown
that it implies that the nonstationary ideal  ${\rm NS}_{\omega_1}$ is precipitous,
and that Chang's Conjecture holds. It is therefore interesting in its own right.
In \cite[Chapter XIII, 1.7]{Sh_P} Shelah showed that $( \dagger )$ is equivalent to a certain reflection
principle. In order to explain this we will introduce some notation.

For countable sets $x$ and $y$, we say that $y$ is an \emph{$\omega_1$-extension} of $x$ if
$x \subseteq y$ and $x \cap \omega_1 = y \cap \omega_1$.
We will write $x \sqsubseteq y$ to say that $y$ is an $\omega_1$-extension of $x$.
Given $S\subseteq [\lambda]^{\omega}$, for some $\lambda \geq \omega_1$,
we will say that $S$ is {\em full} if $S$ is closed under $\omega_1$-extensions.
Shelah \cite{Sh_P} showed that $( \dagger )$ is equivalent to the statement $\ssref$ which
says that, for every $\lambda \geq \omega_2$, the following statement $\ssref (\lambda)$ holds:

\medskip
\begin{quote}{\em
If $S$ is a full stationary subset of $[\lambda]^{\omega}$  then there is
$I\subseteq \lambda$ of cardinality $\omega_1$ such that $\omega_1\subseteq I$
and  $S\cap [I]^{\omega}$ is stationary in $[I]^{\omega}$.}
\end{quote}
\medskip
One may be tempted to conjecture that the assumption that $S$ is full in the above
statement is innocuous and that $\ssref$ is equivalent to $\sref$.
However the first author \cite{Sa} showed that this is not the case, indeed
$\ssref$ is strictly weaker than $\sref$. One of the goals of the present
paper is to show that $\ssref$ nevertheless has many of the consequences
as $\sref$; it implies the Singular Cardinal Hypothesis, the failure of
a weak version of the square principle, etc.

Another topic of this paper has to do with two cardinal properties recently
introduced and studied by Weiss \cite{We}. We first recall the relevant
definitions. Suppose $\kappa$ is a regular cardinal and $\lambda \geq \kappa$.
By $\Fn (\kappa,\lambda,2)$ we denote the set of all partial functions of size
$<\kappa$ from $\lambda$ to $\{ 0,1\}$. A {\em $(\kappa,\lambda)$-tree} is
a family $\F \subseteq \Fn (\kappa,\lambda, 2)$ which is closed under
restrictions and such that for every $u\in [\lambda]^{<\kappa}$ there is $f\in \F$
with $\dom (f)=u$. We denote by ${\lev}_{u}(\F)$ the $u$-level of $\F$,
i.e. the set $\{ f\in \F : \dom (f) =u\}$. A {\em cofinal branch} through
$\F$ is a function $b:\lambda \rightarrow \{ 0,1\}$ such that
$f\restriction u \in \F$, for every $u\in [\lambda]^{<\kappa}$.
A $(\kappa,\lambda)$-tree is called {\em thin} if $\lev_u(\F)$
is of size $<\kappa$, for all $u\in [\lambda]^{<\kappa}$.
A {\em level sequence} of $\F$ is a sequence $\vec{f}=( f_u: u\in [ \lambda ]^{< \kappa} )$
such that $f_u\in {\lev}_u(\F)$ for all $u \in [ \lambda ]^{< \kappa}$.
Given a $(\kappa,\lambda)$-tree and a level sequence $\vec{f}$ of $\F$
we will say that a branch $b$ of $\F$ is {\em ineffable} for $\vec{f}$
if the set $\{ u \in [\lambda]^{<\kappa}: b\restriction u =f_u\}$ is stationary
in $[\lambda]^{ <\kappa}$.  Given a regular cardinal $\kappa \geq \omega_1$ and
$\lambda \geq \kappa$ the two cardinal tree property $\TP (\kappa,\lambda)$ states
that every thin $(\kappa,\lambda)$-tree has a cofinal branch.
We say that $\kappa$ has the strong tree property if $\mrm{TP} ( \kappa , \lambda )$ holds
for every $\lambda \geq \kappa$.
Given $\kappa$ and $\lambda$ as before we let $\ITP (\kappa,\lambda)$ denote the
statement that for every thin  $(\kappa,\lambda)$-tree and a level sequence $\vec{f}$
of $\F$ there is an ineffable branch for $\vec{f}$.
We say that $\kappa$ has the super tree property if $\mrm{ITP} ( \kappa , \lambda )$
holds for every $\lambda \geq \kappa$.
Note that if $\kappa$ is inaccessible then every $(\kappa,\lambda)$-tree is thin.
With this in mind we can now reinterpret classical results of Jech \cite{Jec}
and Magidor \cite{Mag}. Namely, in our terminology, Jech \cite{Jec}
showed that an uncountable cardinal is strongly compact if and only if it is inaccessible
and has the strong tree property.
Similarly, Magidor \cite{Mag} showed that an uncountable cardinal $\kappa$
is supercompact if and only if it is inaccessible and has the super tree property.
These results are analogous to the classical
reformulation of weak compactness which states that an uncountable cardinal
$\kappa$ is weakly compact if and only if it is inaccessible and the usual tree property
holds for $\kappa$, see for instance \cite{Kanamori}.
Since all known proofs of the consistency of strong
forcing axioms require supercompact cardinals it was natural to expect
that they would imply these two cardinal properties for $\kappa =\omega_2$.
This was indeed confirmed by Weiss \cite{We} who showed that the Proper
Forcing Axiom ($\PFA$) implies that $\omega_2$ has the super tree property.
Moreover, Viale and Weiss \cite{VW} showed that if
the universe $V$ is obtained by forcing over some inner model $M$ by a forcing notion which has
the $\kappa$-chain condition and the $\kappa$-approximation property,
then if $\kappa$ has the strong tree property in $V$, it also has the strong tree property in $M$.
If, moreover, the forcing notion is proper then the same holds
for the super tree property.
Since all known methods for producing $\PFA$ start from
an inaccessible cardinal $\kappa$ in some universe $M$ and produce a generic
extension by a forcing notion which has the above property and in which
$\kappa$ becomes $\omega_2$, it follows that they require at least a strongly
compact cardinal. We will show that the $\SR$ together with $\MAC$
implies the super tree property of $\omega_2$
and that $\SSR$ together with $\MAC$ implies the strong tree property of $\omega_2$.
This suggests that $\SR + \MAC$ should
have the consistency strength of a supercompact cardinal whereas $\SSR +\MAC$
should have the strength of a strongly compact cardinal.
We also show that $\SSR + \MAC$ does not imply the super tree property of $\omega_2$.

This paper is organized as follows.
In Section \ref{sec:preliminaries} we present notation and basic facts used in this paper.
In Section \ref{sec:wsquare} we prove that $\ssref$ implies the failure of weak square principles.
In Section \ref{sec:itp_tp} we prove that $\sref + \MAC$ and $\ssref + \MAC$
imply the super and the strong tree properties, respectively.
Finally in Section \ref{sec:sch} we prove that $\ssref$ implies $\mrm{SCH}$.


\section{Preliminaries} \label{sec:preliminaries}

In this section we present notation and basic facts used in this paper.
For a set $A$ of ordinals let $\ltpt (A)$ be the set of all limit points in $A$.
Moreover let
$\bsup (A) = \sup \{ \alpha + 1 : \alpha \in A \}$.
We often use $\bsup$ instead of $\sup$ since it slightly simplifies our arguments.
For an ordinal $\lambda$ and a regular cardinal $\kappa < \lambda$
let $E^\lambda_\kappa = \{ \alpha < \lambda : \cof ( \alpha ) = \kappa \}$.

Let $A$ be a set and $F$ be a function from $[A]^{< \omega}$ to $A$.
We say that $x \subseteq A$ is closed under $F$ if $F(a) \in x$, for all $a \in [x]^{< \omega}$.
For each $x \subseteq A$ let $\clsr_F (x)$ be the closure of $x$ under $F$,
i.e. the smallest subset $y$ of $A$ which contains $x$ and is closed under $F$.

Let $\kappa$ be a regular uncountable cardinal and $A$ be a set including $\kappa$.
Recall that a subset $C$ of $[A]^{< \kappa}$ is said to be \emph{club}
if and only if it is $\subseteq$-cofinal in $[A]^{< \kappa}$,
and closed under unions of  $\subseteq$-increasing sequence  of length $< \kappa$.
$S \subseteq [A]^{< \kappa}$ is said to be \emph{stationary}
if it intersects all club subsets of $[A]^{< \kappa}$.
We often use the well-known fact that $S$ is stationary in $[A]^{< \kappa}$ if and only if
for any function $F: [A]^{< \omega} \to A$ there exists $x \in S$ which is closed
under $F$ and such that $x \cap \kappa \in \kappa$.
We say that $X \subseteq [A]^\kappa$ is stationary (or club) if $X$ is stationary (or club)
in $[A]^{< \kappa^+}$.

For a set $A$ and a limit ordinal $\eta$ we say that $A$ is
\emph{internally approachable of length $\eta$} if there exists a $\subseteq$-increasing
sequence $( x_\xi : \xi < \eta )$
such that $\bigcup_{\xi < \eta} x_\xi = A$
and such that $( x_\xi : \xi < \zeta  ) \in A$, for all $\zeta  < \eta$.

Suppose that $\frak A=(A,\unlhd,\ldots)$ is a structure in a countable first order language
and $\unlhd$ is a well ordering of $A$. Then $\frak A$ has definable Skolem functions.
For each $X \subseteq A$ let $\skull^\mcal{\frak A} (X )$ be the Skolem hull
of $X$ in $\frak A$, i.e.~$\skull^\mcal{\frak A} (X)$ is the smallest $M$ elementary submodel
of  $\frak A$ such that $X \subseteq M$.
We say that a structure $\frak A=(A,\ldots)$ is an \emph{expansion} of a structure $\frak A'=(A,\ldots)$
if $\frak A$ is obtained by adding countable many constants, functions and predicates
to $\frak A'$. We use the following fact.

\begin{fact}[folklore] \label{fact:skull_sup}
Let $\theta$ be a regular uncountable cardinal and $\unlhd$ a well ordering of $H_\theta$.
Let  $\frak A = (H_\theta,\in, \unlhd, \ldots)$ be a structure in a countable language and
suppose $M$ is an elementary submodel of  $\frak A$ and $\lambda$ is a regular uncountable cardinal
with $\lambda \in M$. Let $\delta= \sup ( M \cap \lambda )$. Then
$\skull^{\frak A} ( M \cup \delta ) \cap \lambda = \delta$.
\end{fact}

\begin{proof}
It suffices to prove that $\skull^\frak{A} ( M \cup \delta ) \cap \lambda \subseteq \delta$.
Take an arbitrary $\alpha \in \skull^\frak{A} ( M \cup \delta ) \cap \lambda$.
Then there are a formula $\varphi ( v_0 , v_1 , v_2 )$, $b \in [ \delta ]^{< \omega}$
and $p \in M$ such that $\alpha$ is the unique element with
$\frak{A} \models \varphi [ \alpha , b , p ]$.
Take $\gamma \in M \cap \lambda$ with $b \in [ \gamma ]^{< \omega}$,
and for each $a \in [ \gamma ]^{< \omega}$ let $h(a)$ be the least $\xi < \lambda$
with $\frak{A} \models \varphi [ \xi , a , p ]$ if such $\xi$ exists.
Then $h$ is a partial function from $[ \gamma ]^{< \omega}$ to $\lambda$,
and $h \in M$ by the definability of $h$ and the elementarily of $M$.
Then
\[
\alpha = h(b) < \sup ( \mrm{ran} (h) ) \in M \cap \lambda \; .
\]
Hence $\alpha \in \delta$.
\end{proof}

Next we give our notation and facts relevant to the singular cardinal combinatorics.
Recall that $\mrm{SCH}$ is the statement that
$\lambda^{\cof ( \lambda )} = \lambda^+$ for all singular cardinals $\lambda$
with $2^{\cof ( \lambda )} < \lambda$.
We say that $\mrm{SCH}$ fails at a singular cardinal $\lambda$
if $2^{\cof ( \lambda )} < \lambda$, and $\lambda^{\cof ( \lambda )} > \lambda^+$.
We use the following well-known theorem:

\begin{fact}[Silver \cite{Sil}] \label{fact:Silver}
Suppose $\lambda$ is the least singular cardinal
at which $\mrm{SCH}$ fails.
Then $\cof ( \lambda ) = \omega$.
\end{fact}

We also use Shelah's PCF theory. Since we will only be working with singular
cardinals of cofinality $\omega$ we make the relevant definitions only in this case.
Let $\vec{\lambda} = ( \lambda_n : n \in \omega )$ be a strictly increasing
sequence of regular cardinals, and let $\lambda = \sup_{n \in \omega} \lambda_n$.
We let $\prod \vec{\lambda}$ denote $\prod_{n \in \omega} \lambda_n$.
For a set $x$ of ordinals with $|x| < \lambda_0$
let $\chi_x^{\vec{\lambda}} \in \prod \vec{\lambda}$ be the characteristic function of $x$,
i.e.~$\chi_x^{\vec{\lambda}} (n) = \bsup ( x \cap \lambda_n )$ for each $n \in \omega$.
We will omit the superscript $\vec{\lambda}$ in $\chi_x^{\vec{\lambda}}$ if it is clear
from the context.

For functions $f , g : \omega \to \mrm{On}$ we use the following notation:
\begin{eqnarray*}
f < g & \defarrow & \forall n  \; f(n) < g(n) \\
f <^* g & \defarrow & \exists m \; \forall n \geq m  \; f(n) < g(n) \\
f =^* g & \defarrow & \exists m \; \forall n \geq m  \; f(n) = g(n)
\end{eqnarray*}
Moreover for $m < \omega$ we use the following:
\begin{eqnarray*}
f <_m g & \defarrow & \forall n \geq m  \; f(n) < g(n) \\
f =_m g & \defarrow & \forall n \geq m  \; f(n) = g(n)
\end{eqnarray*}
$f \leq g$, $f \leq^* g$ and $f \leq_m g$ are defined in the same way as
$f < g$, $f <^* g$ and $f <_m g$.

A $<^*$-increasing cofinal sequence in $\prod \vec{\lambda}$ of length $\lambda^+$
is called a \emph{scale} on $\vec{\lambda}$.
A scale $( f_\beta : \beta < \lambda^+ )$ is called
a \emph{better scale} if for any $\alpha < \lambda^+$ of uncountable cofinality
there exists a club $C \subseteq \alpha$ and $\sigma : C \to \omega$ such that
for any $\beta , \gamma \in C$ with $\beta < \gamma$
we have $f_\beta <_{\max \{ \sigma ( \beta ) , \sigma ( \gamma ) \}} f_\gamma$.
We use the following fact.

\begin{fact}[Shelah \cite{Sh_CA}] \label{fact:better_scale}
Suppose that $\lambda$ is a singular cardinal of cofinality $\omega$ such that
$\mu^\omega < \lambda$ for all $\mu < \lambda$ and such that $\lambda^\omega > \lambda^+$.
Then there exists a strictly increasing sequence of regular cardinals of length $\omega$
which converges to $\lambda$ and on which a better scale exists.
\end{fact}


\section{Failure of weak square} \label{sec:wsquare}

It is known, due to Veli\v{c}kovi\'{c} \cite{Vel}, that $\sref$ implies
the failure of $\square ( \lambda )$ for all regular $\lambda \geq \omega_2$.
Recall that $\square ( \lambda )$ says that there is a
sequence $( C_\alpha : \alpha \in \ltpt ( \lambda ) )$
such that
\begin{renumerate}
\item $C_\alpha$ is a club subset of $\alpha$, for all $\alpha$,
\item if $\alpha \in \ltpt ( C_\beta )$, then $C_\alpha = C_\beta \cap \alpha$,
\item there are no club $C \subseteq \lambda$ with
$C \cap \alpha = C_\alpha$ for all $\alpha \in \ltpt (C)$.
\end{renumerate}

\noindent
In this section we prove that $\ssref$ also denies $\square ( \lambda )$
for all regular $\lambda \geq \omega_2$:

\begin{theorem} \label{thm:weak_square}
Assume that $\lambda$ is a regular cardinal $\geq \omega_2$
and that $\ssref ( \lambda )$ holds.
Then $\square ( \lambda )$ fails.
\end{theorem}

Our proof is based on that in \cite{Vel}.
To prove Theorem \ref{thm:weak_square} we need several preliminaries.

First we give a modification of $\ssref$, which is also used in Section \ref{sec:sch}.
For countable sets $x$ and $y$ we write $x \sqsubseteq^* y$ if
\begin{renumerate}
\item $x \sqsubseteq y$,
\item $\bsup (x) = \bsup (y)$,
\item $\bsup ( x \cap \gamma ) = \bsup ( y \cap \gamma )$ for all $\gamma \in E^\lambda_{\omega_1} \cap x$.
\end{renumerate}
Given $X \subseteq [ \lambda ]^\omega$, for some $\lambda \geq \omega_1$,
we say that $X$ is \emph{weakly full} if $X$ is upward closed under $\sqsubseteq^*$.

\begin{lemma} \label{lem:ssr_sup}
Assume that $\lambda \geq \omega_2$ and
that $\ssref ( \lambda )$ holds.
Then for any weakly full stationary $X \subseteq [ \lambda ]^\omega$
there exists $I \in [ \lambda ]^{\omega_1}$ including $\omega_1$ such that
$X \cap [J]^\omega$ is stationary for all $J \subseteq \lambda$ such that $I\subseteq J$
and $\bsup (J) = \bsup (I)$.
\end{lemma}

\begin{proof}
Let $X$ be a weakly full stationary subset of $[ \lambda ]^\omega$.
Take a sufficiently large regular cardinal $\theta$ and a well-ordering $\unlhd$ of $H_\theta$,
and let $\frak{A} = ( H_\theta , \in , \unlhd , \lambda )$.
Let $Y$ be the set of all $y \in X$ with $\skull^\frak{A} (y) \cap \lambda = y$,
and let $\overline{Y}$ be the upward closure of $Y$ under $\sqsubseteq$.
By $\ssref ( \lambda )$ there is $I' \in [ \lambda ]^{\omega_1}$
such that $\omega_1 \subseteq I'$ and
$\overline{Y} \cap [I']^\omega$ is stationary.
Let $I$ be one of such $I'$ with the least $\bsup$.
We show that $I$ witnesses the lemma for $X$.
To this end, we make a preliminary definition. Let
\[
Z_0 = \{ z \in [I]^\omega : \exists y \in Y , \; y \sqsubseteq z \,\wedge\,
\bsup (y) = \bsup (z) \} \; .
\]

\begin{claim*}
$Z_0$ is stationary in $[I]^\omega$.
\end{claim*}

\noindent
\emph{Proof of Claim}.
Assume not. Then
\[
Z = \{ z \in [ I ]^\omega : \exists y \in Y , \;
y \sqsubseteq z \,\wedge\, \bsup (y) < \bsup (z) \}
\]
is stationary in $[I]^\omega$.
For each $z \in Z$, choose $y_z \in Y$ such that $y_z \sqsubseteq z$ and
$\bsup ( y_z ) < \bsup ( z )$,
and let $\beta_z = \min ( z \setminus \bsup ( y_z ) )$. Note that $\beta_z \geq \omega_1$.
By Fodor's lemma we can find $\beta$ such that
$Z' = \{ z \in Z : \beta_z = \beta \}$ is stationary in $[I]^\omega$.
Let $I' = I \cap \beta$.
Then $\{ z \cap \beta : z \in Z' \}$ is stationary in $[I']^\omega$.
Moreover $z \cap \beta \in \overline{Y}$ for each $z \in Z'$ because $z \cap \beta \sqsupseteq y_z$.
So $\overline{Y} \cap [I']^\omega$ is stationary. Note also that $\omega_1\subseteq I'$ and $|I'| = \omega_1$.
But $\bsup (I') < \bsup (I)$. This contradicts the choice of $I$.
\hfill $\square_{\mrm{Claim}}$

\bigskip

Now we prove that $I$ witnesses the lemma for $X$.
Take an arbitrary $J \subseteq \lambda$ with $I \subseteq J$ and $\bsup (J) = \bsup (I)$.
Let
\[
Z_1 = \{ z \in [J]^\omega : z \cap I \in Z_0 \,\wedge\, \bsup ( z \cap I ) = \bsup (z) \,\wedge\,
\skull^\frak{A} (z) \cap \omega_1 = z \cap \omega_1 \} \; .
\]
Then $Z_1$ is stationary in $[J]^\omega$ because $Z_0$ is stationary,
$\bsup (J) = \bsup (I)$, and $\omega_1 \subseteq J$.
We show that $Z_1 \subseteq X$.
In order to see this, take an arbitrary $z \in Z_1$. We prove that $z \in X$.
First we can take $y \in Y$ with $y \sqsubseteq z$ and $\bsup (y) = \bsup (z)$.
Recall that $Y \subseteq X$ and that $X$ is closed under $\sqsubseteq^*$.
So it suffices to prove that $y \sqsubseteq^* z$.
For this all we have to show is that $\bsup ( y \cap \gamma ) \geq \bsup ( z \cap \gamma )$
for every $\gamma \in E^\lambda_{\omega_1} \cap y$.
Suppose that $\gamma \in E^\lambda_{\omega_1} \cap y$.
Let $M = \skull^\frak{A} (y)$ and $N = \skull^\frak{A} (z)$.
Note that $M \cap \omega_1 = y \cap \omega_1 = z \cap \omega_1 = N \cap \omega_1$.
Then, since $\gamma \in M \subseteq N$ and $\cof ( \gamma ) = \omega_1$,
we have that $\bsup ( M \cap \gamma ) = \bsup ( N \cap \gamma )$.
Moreover $\bsup ( z \cap \gamma ) \leq \bsup ( N \cap \gamma )$,
and $\bsup ( y \cap \gamma ) = \bsup ( M \cap \gamma )$ by the definition of $Y$.
Hence $\bsup ( y \cap \gamma ) \geq \bsup ( z \cap \gamma )$.
\end{proof}

Next we present a game which will be used to construct
a weakly full stationary set.
Let $\lambda$ be a regular cardinal $\geq \omega_2$.
For a function $F : [ \lambda ]^{< \omega} \to \lambda$
let $G_1 ( \lambda , F )$ be the following game of length $\omega$:
\[
\begin{array}{c | c| c|c|c|c}
\mrm{I} & \alpha_0 \ \phantom{\beta_0} \ \gamma_0 &
\alpha_1 \ \phantom{\beta_1} \ \gamma_1 & \ \ \cdots \ \ &
\alpha_n \ \phantom{\beta_n} \ \gamma_n & \ \ \cdots \ \ \\
\hline
\mrm{II} & \phantom{\alpha_0} \ \beta_0 \ \phantom{\gamma_0} &
\phantom{\alpha_1} \ \beta_1 \ \phantom{\gamma_1} & \ \ \cdots \ \ &
\phantom{\alpha_n} \ \beta_n \ \phantom{\gamma_n} & \ \ \cdots \ \ \\
\end{array}
\]
I and II in turn choose ordinals $< \lambda$.
In the $n$-th stage, first I chooses $\alpha_n$, then II chooses $\beta_n$,
and then I again chooses $\gamma_n > \alpha_n , \beta_n$ of cofinality $\omega_1$.
I wins if
\[
\clsr_F ( \{ \gamma_n : n \in \omega \} ) \cap [ \alpha_m , \gamma_m )
= \emptyset
\]
for every $m \in \omega$.
Otherwise, II wins.

\begin{lemma} \label{lem:game1}
Let $\lambda$ be a regular cardinal $\geq \omega_2$ and
$F$ be a function from $[ \lambda ]^{< \omega}$ to $\lambda$.
Then I has a winning strategy for the game $G_1 ( \lambda , F )$.
\end{lemma}

\begin{proof}
Since $G_1(\lambda,F)$ is an open game for player I, by the Gale-Stewart theorem, one
of the players has a winning strategy. Assume towards contradiction that
II has a winning strategy, say $\tau$.
We will find a play
$( \alpha_n , \beta_n , \gamma_n : n \in \omega )$ in which
II follows $\tau$, but which is won by I.

Let $\theta$ be a sufficiently large regular cardinal.
First, build an $\in$-chain $\{ M_n: n<\omega\}$ of elementary submodels
of $H_\theta$ containing $F$ and $\tau$ as elements and such that
$\gamma_n= M_n\cap \lambda$ is an ordinal $<\lambda$ of cofinality $\omega_1$.
Let $x = \clsr_F ( \{ \gamma_n : n \in \omega \} )$
and $\alpha_n = \sup ( x \cap \gamma_n )$, for each $n$.
Note that $\alpha_n < \gamma_n$, since $x$ is countable and $\cof ( \gamma_n ) = \omega_1$.
Finally let $( \beta_n : n \in \omega )$ be a sequence of II's moves
according $\tau$ against $( \alpha_n , \gamma_n : n \in \omega )$.
Note that $\beta_n < \gamma_n$ since
$\alpha_0 , \gamma_0 , \dots , \alpha_{n-1} , \gamma_{n-1} , \alpha_n \in
M_n$ and $M_n$ is an elementary submodel of $H_\theta$ containing $\tau$.
Now $( \alpha_n , \beta_n , \gamma_n : n \in \omega )$
is a legal play of $G_1 ( \lambda , F )$ in which II has followed $\tau$.
However,  $x \cap [ \alpha_n , \gamma_n ) = \emptyset$, for each $n$,
by the definition of the $\alpha_n$.
Therefore I wins this play, a contradiction. It follows that
I has a winning strategy in $G_1(\lambda,F)$ as required.
\end{proof}

Now we prove Theorem \ref{thm:weak_square}.

\begin{proof}[Proof of Theorem \ref{thm:weak_square}]
Assuming that $\square ( \lambda )$ holds, we prove that $\ssref ( \lambda )$ fails.
Let $\vec{C} = ( C_\alpha : \alpha \in \ltpt ( \lambda ) )$
be a $\square ( \lambda )$-sequence.
Let $X$ be the set of all $x \in [ \lambda ]^\omega$ which have limit order type
and there is $\xi <\bsup(x)$ such that:

\begin{enumerate}
\item $\sup (x\cap C_{\bsup (x)}) \leq \xi$,
\item $\cof (\min (x \setminus \beta))=\omega_1$, for all $\beta \in C_{\bsup (x)} \setminus \xi$.
\end{enumerate}

Here note that $X$ is weakly full.
So it suffices to prove the following claims.

\begin{claim}
$X$ is stationary in $[ \lambda ]^\omega$.
\end{claim}

\noindent
\begin{proof}
Take an arbitrary function $F : [ \lambda ]^{< \omega} \to \lambda$.
We find $x \in X$ closed under $F$.
By Lemma \ref{lem:game1} fix a winning strategy
$\tau$ of I for $G_1 ( \lambda , F )$.
Moreover let $C$ be the set of all limit ordinals $\beta < \lambda$ closed under $\tau$ and $F$.
Note that $C$ is club in $\lambda$.
Then, since $\vec{C}$ is a $\square ( \lambda )$-sequence,
there exists $\delta \in \ltpt (C) \cap E^\lambda_\omega$
such that $C \cap \delta \setminus C_\delta$ is unbounded in $\delta$.
Take a strictly increasing sequence $( \delta_n : n \in \omega )$
in $C \cap \delta \setminus C_\delta$ which is cofinal in $\delta$.
For each $n \in \omega$ we can take $\beta_n < \delta_n$ such that
$[ \beta_n , \delta_n ) \cap C_\delta = \emptyset$
because $\delta_n$ is a limit ordinal which is not in $C_\delta$.
Then let $( \alpha_n , \gamma_n : n \in \omega )$
be a sequence of I's moves according to $\tau$ against $( \beta_n : n \in \omega )$.
Moreover let
$x = \clsr_F ( \{ \gamma_n : n \in \omega \} )$.
It suffices to prove that $x \in X$.

To see this, first note that $\bsup (x) = \delta$ because $\delta$ is closed under $F$.
Next note that $\alpha_{n+1} < \delta_n$ for each $n \in \omega$
because $\delta_n$ is closed under $\tau$.
Moreover $C_\delta \cap \delta_{n+1} \subseteq \beta_{n+1} \subseteq \gamma_{n+1}$
by the choice of $\beta_{n+1}$.
Hence $C_\delta \cap [ \delta_n , \delta_{n+1} ) \subseteq [ \alpha_{n+1} , \gamma_{n+1} )$
for every $n \in \omega$.
Note that $x \cap [ \alpha_{n+1} , \gamma_{n+1} ) = \emptyset$ for each $n \in \omega$
because I wins with the play $( \alpha_n , \beta_n , \gamma_n : n \in \omega )$.
Thus $x \cap C_\delta \subseteq \delta_0$.
Moreover $\min ( x \setminus \beta ) = \gamma_{n+1}$
for all $\beta \in C_\delta \cap [ \delta_n , \delta_{n+1} )$,
and $\cof ( \gamma_{n+1} ) = \omega_1$ by the rule of $G_1 ( \lambda , F )$.
Therefore $\xi = \delta_0$ witnesses that $x \in X$.
\end{proof}

\bigskip

\begin{claim}
The conclusion of Lemma \ref{lem:ssr_sup} fails for $X$.
\end{claim}

\begin{proof} It suffices to prove that $X \cap [ \delta ]^\omega$ is non stationary
for every ordinal $\delta \in \lambda \setminus \omega_1$.
If $\delta$ is a successor ordinal, then $X \cap [ \delta ]^\omega$
is clearly non stationary.
Next suppose that $\cof ( \delta ) = \omega$.
Let $Z_0$ be the set of all $z \in [ \delta ]^\omega$
such that $\bsup (z) = \delta$ and such that
$z \cap C_\delta$ is unbounded in $\delta$.
Then $Z_0$ is club in $[ \delta ]^\omega$, and $X \cap Z_0 = \emptyset$.
Thus $X \cap [ \delta ]^\omega$ is non stationary.

Finally suppose that $\cof ( \delta ) > \omega$.
Let $Z_1$ be the set of all $z \in [ \delta ]^\omega$
such that $\bsup (z) \in \ltpt ( C_\delta )$ and
such that $z \cap C_\delta$ is unbounded in $\bsup (z)$.
Then $Z_1$ is club in $[ \delta ]^\omega$.
Here note that $z \cap C_{\bsup (z)}$ is unbounded in $\bsup (z)$ for each $z \in Z_1$
because $C_\delta \cap \bsup (z) = C_{\bsup (z)}$.
So $X \cap Z_1 = \emptyset$.
\end{proof}
This concludes the proof of Theorem \ref{thm:weak_square}.
\end{proof}
\setcounter{claim}{0}


\section{$\mrm{ITP}$ and $\mrm{TP}$} \label{sec:itp_tp}

In this section we prove that $\sref + \mrm{MA}_{\omega_1} ( \mrm{Cohen} )$
implies the super tree property at $\omega_2$ and that
$\ssref + \mrm{MA}_{\omega_1} ( \mrm{Cohen} )$ implies the strong tree property
at $\omega_2$.
Here note that the tree property at $\omega_2$ implies the failure of $\mrm{CH}$
and that $\sref$ and $\ssref$ are consistent with $\mrm{CH}$.
So $\sref$ or $\ssref$ alone does not imply the super or strong tree property at $\omega_2$,
respectively.
We also prove that $\ssref + \mrm{MA}_{\omega_1} ( \mrm{Cohen} )$ does not imply
the super tree property at $\omega_2$.

\newpage

\begin{theorem} \label{thm:itp_tp}
${}$
\begin{aenumerate}
\item[(a)] If $\sref$ and $\mrm{MA}_{\omega_1} ( \mrm{Cohen} )$ hold,
then $\omega_2$ has the super tree property.
\item[(b)] If $\ssref$ and $\mrm{MA}_{\omega_1} ( \mrm{Cohen} )$ hold,
then $\omega_2$ has the strong tree property.
\end{aenumerate}
\end{theorem}

Before we proof Theorem \ref{thm:itp_tp}, we introduce some notation.
For an ordinal $\lambda \geq \omega_2$ and a set $M$ let
\[
u^\lambda_M = \bigcup \left( [ \lambda ]^{\omega_1} \cap M \right) \; .
\]
We will omit the superscript $\lambda$ in $u^\lambda_M$ if it is clear from the context.
The following is a key lemma.

\begin{lemma} \label{lem:ma_tree}
Assume $\mrm{MA}_{\omega_1} ( \mrm{Cohen} )$.
Let $\lambda$ be an ordinal $\geq \omega_2$ and $\F$ be a thin
$( \omega_2 , \lambda )$-tree.
Let $\theta$ be a sufficiently large regular cardinal.
Then there are stationary many $M \in [ H_\theta ]^\omega$ such that
for all $f \in \mrm{lev}_{u_M} ( \F )$ exactly one of the following holds:
\begin{enumerate}
\item There exists $b \in {}^\lambda 2 \cap M$ with $b \rst u_M = f$.
\item There exists $u \in [ \lambda ]^{\omega_1} \cap M$ with
$f \rst u \notin M$.
\end{enumerate}
\end{lemma}

To prove this lemma we further need two lemmas.
First, let $\theta$ be a regular cardinal $\geq \omega_2$.
For a function $F : [ H_\theta ]^{< \omega} \to H_\theta$ and $\xi < \omega_1$,
let $G_2 ( \theta , F , \xi )$ be the following game of length $\omega$:
\[
\begin{array}{c|c|c|c|c|c}
\mrm{I} & J_0 \phantom{K_0} & J_1 \phantom{K_1}& \ \ \cdots \ \ & J_n \phantom{K_n}& \ \ \cdots \ \ \\
\hline
\mrm{II} & \phantom{J_0} K_0 & \phantom{J_1} K_1 & \ \ \cdots \ \ & \phantom{J_n} K_n & \ \ \cdots \ \ \\
\end{array}
\]
In the $n$-th stage first I chooses $J_n \in [ H_\theta ]^{\omega_1}$
and then II chooses $K_n \in [ H_\theta ]^{\omega_1}$ with $K_n \supseteq J_n$.
II wins if and only if
\[
\clsr_F ( \xi \cup \{ K_n : n \in \omega \} ) \cap \omega_1 ~=~ \xi \; .
\]

\begin{lemma} \label{lem:game2}
Let $\theta$ be a regular cardinal $\geq \omega_2$.
Then for any function $F : [ H_\theta ]^{< \omega} \to H_\theta$
there exists $\xi < \omega_1$ such that {\rm II} has a winning strategy
for $G_2 ( \theta , F , \xi )$.
\end{lemma}

\begin{proof}
Take an arbitrary function $F : [ H_\theta ]^{< \omega} \to H_\theta$.
Towards contradiction assume that II does not have a winning strategy for
$G_2 ( \theta , F , \xi )$, for any $\xi < \omega_1$.
Since each $G_2 ( \theta , F , \xi )$ is an open-closed game,
I has a winning strategy, say $\tau_\xi$, in $G_2 ( \theta , F , \xi )$, for each $\xi$.
Let $\vec{\tau} = ( \tau_\xi : \xi < \omega_1 )$.
Take a sufficiently large regular cardinal $\mu$
and a countable elementary submodel $M$ of $H_\mu$
containing $\theta,F$ and $\vec{\tau}$. Let $\zeta = M \cap \omega_1$.
By induction on $n$ let
\begin{eqnarray*}
J_n & = & \tau_\zeta ( ( K_m : m < n ) ) \; ,\\
K_n & = & {\textstyle \bigcup_{\xi < \omega_1}} \tau_\xi ( ( K_m : m < n ) ) \;.
\end{eqnarray*}
Then $( J_n , K_n : n \in \omega )$ is a legal play
of $G_2 ( \theta , F , \zeta )$ in which I has moved according to $\tau_\zeta$.
Here note that $K_n \in M$, for each $n$, by the elementarity of $M$.
Moreover $\zeta = M \cap \omega_1$ and $M$ is closed under $F$. Hence
$\clsr_F ( \zeta \cup \{ K_n : n \in \omega \} ) \subseteq M$,
and thus $\clsr_F ( \zeta \cup \{ K_n : n \in \omega \} ) \cap \omega_1 = \zeta$.
Therefore II wins the play $( J_n , K_n : n \in \omega )$ in $G_2(\theta,F,\zeta)$,
which  contradicts the fact that $\tau_\zeta$ is a winning strategy of I.
\end{proof}

The second one is a lemma on very thin $( \omega_2 , \lambda )$-trees.

\begin{lemma} \label{lem:very_thin_tree}
Let $\lambda$ be an ordinal $\geq \omega_2$ and $\mathscr{F}$ be an $( \omega_2 , \lambda )$-tree
such that $\mrm{lev}_u ( \F )$ is countable for every $u \in [ \lambda ]^{\leq \omega_1}$.
Then there is a countable subset $\mathscr{B}$ of  $ {}^\lambda 2$
and a club $C$ in  $[ \lambda ]^{\leq \omega_1}$ such that
for any $u \in C$ and $f \in \mrm{lev}_u ( \F )$ there is a unique $b \in \B$
with $b \rst u = f$.
\end{lemma}

\begin{proof}
Let $\theta$ be a sufficiently large regular cardinal
and $W_0$ be the set of all elementary submodels $K$ of  $H_\theta$
which have cardinality $\aleph_1$ and are  internally approachable of length $\omega_1$.
Note that $W_0$ is stationary in $[ H_\theta ]^{\leq \omega_1}$.
For each $K \in W_0$ we can take $x_K \in [ K \cap \lambda ]^{\leq \omega_1} \cap K$ such that
$f_0 \rst x_K \neq f_1 \rst x_K$ for any distinct $f_0 , f_1 \in \mrm{lev}_{K \cap \lambda} ( \F )$.
This is because $\mrm{lev}_{K \cap \lambda} ( \F )$ is countable.
By the Pressing Down Lemma, there is $x \in [ \lambda ]^{\leq \omega_1}$ such that
$W_1 = \{ K \in W_0 : x_K = x \}$ is stationary.
Let $U = \{ K \cap \lambda : K \in W_1 \}$.
Note that $U$ is stationary in $[ \lambda ]^{\leq \omega_1}$.
Let $A$ be the set of all $h \in \mrm{lev}_x ( \F )$
such that for every $u \in U$ there is $f \in \mrm{lev}_u ( \F )$ with $f \rst x = h$.
For each $h \in A$ and $u \in U$
let $f^h_u$ be the unique element of $\mrm{lev}_u ( \F )$ with $f^h_u \rst x = h$.
Here note that if $h \in A$, and $u , v \in U$,
then $f^h_u \rst ( u \cap v ) = f^h_v \rst ( u \cap v )$.
For each $h \in A$ let $b_h = \bigcup \{ f^h_u : u \in U \} \in {}^\lambda 2$,
and let $\B = \{ b_h : h \in A \}$. 
Clearly $\B$ is a countable subset of ${}^\lambda 2$.
Note that $b_h \rst u \in \mrm{lev}_u ( \F )$ for all $u \in [ \lambda ]^{\leq \omega_1}$
since $U$ is $\subseteq$-cofinal in $[ \lambda ]^{\leq \omega_1}$.
Let $D$ be the collection of all elementary submodels $K$ of $H_\theta$ which have size
$\aleph_1$ and contain all the relevant objects. Then $D$ is a club in $[H_\theta]^{\omega_1}$
and $C=\{\ K\cap \lambda : K\in D\}$ is a club in $[\lambda]^{\omega_1}$.
We claim that $\B$ and $C$ are as desired.
In order to see this, fix $K\in D$ and $f \in \mrm{lev}_{K \cap \lambda} ( \F )$.
Let $h = f \rst x$.
Then $h \in \mrm{lev}_x ( \F ) \subseteq K$,
and $f \rst u$ witnesses that there is $g \in \mrm{lev}_u ( \F )$ with $g \rst x = h$,
for every $u \in U \cap K$.
So it follows from the elementarity of $K$ that $h \in A$.
Moreover $f \rst u = f^h_u$ for all $u \in U \cap K$,
and $\bigcup ( U \cap K ) = K \cap \lambda$ by the elementarity of $K$.
Therefore $f = \bigcup_{u \in U \cap K} f^h_u = b_h \rst ( K \cap \lambda )$.
\end{proof}


\begin{proof}[Proof of Lemma \ref{lem:ma_tree}]
Take an arbitrary function $F : [ H_\theta ]^{< \omega} \to H_\theta$.
We find a countable elementary submodel  $M$ of  $[ H_\theta ]$ closed under $F$ such that
for any $f \in \mrm{lev}_{u_M} ( \F )$ either (1) or (2) in Lemma \ref{lem:ma_tree} holds.
Let $\unlhd$ be a well-ordering of $H_\theta$.
By changing $F$ if necessary, we may assume that if a subset $M$ of $H_\theta$ is closed under $F$, then
$M$ is an elementary submodel of  $(H_\theta,\in,\unlhd)$ and contains $\lambda$ and $\F$.
By Lemma \ref{lem:game2} let $\xi < \omega_1$ be such that II has a winning strategy, say
$\tau$, for $G_2 ( \theta , F , \xi )$.
Moreover take a sufficiently large regular cardinal $\mu$ and a countable elementary submodel
$N$ of $H_\mu$ containing all the relevant objects.
The desired $M$ will be a subset of $N$ and will be obtained by applying
$\mrm{MA}_{\omega_1} ( \mrm{Cohen} )$ to an appropriate poset.

Let $\Pcal$ be the set of  partial plays of the form
$p = \langle J_0^p,K_0^p,\ldots, J_{n_p-1}^p,K_{n_p-1}^p\rangle$ in the game $G_2(\theta,F,\xi)$
in which II follows his winning strategy $\tau$.
We call the integer $n_p$ the {\em length} of $p$.
Moreover, let $u^p_i= K_i^p \cap \lambda$, for all $i<n_p$.
We order $\Pcal$ by reverse end extension.
We will apply $\mrm{MA}_{\omega_1} ( \mrm{Cohen} )$ to the poset $\Pcal_N = \Pcal \cap N$.
Note that, since $N$ is countable, so is  $\Pcal_N$.

Given $K \in [ H_\theta ]^{\omega_1} \cap N$, let
\[
D_K = \{ p \in \Pcal_N : K \subseteq K_i^p, \mbox{ for some }  i < n_p  \} \; .
\]
Then $D_K$ is dense in $\Pcal_N$, for all such $K$.
Next, for  $u \in [ \lambda ]^{\leq \omega_1}$ let
$( f^u_\zeta : \zeta < \omega_1 )$ be the $\unlhd$-least enumeration
of $\mrm{lev}_u ( \F )$, and let $A(u) = \{ f^u_\zeta : \zeta < \xi \}$.
Note that $( A(u) : u \in [ \lambda ]^{\leq \omega_1} ) \in N$.
Then for each $f \in \mrm{lev}_{u_N} ( \F )$ let
\[
E_f = \{ p \in \Pcal_N :
f \rst  u^p_i \notin A(u^p_i),  \mbox{ for some }  i < n_p  \} .
\]

\begin{claim*}
Suppose $f \in \mrm{lev}_{u_N} ( \F )$ and $E_f$ is not dense in $\Pcal_N$.
Then there is $b \in {}^\lambda 2 \cap N$
such that $b \rst u_N = f$.
\end{claim*}

\noindent
\begin{proof} Let  $f\in \mrm{lev}_{u_N} ( \F )$ be such that $E_f$ is not dense in $\Pcal_N$.
We find $b\in {}^\lambda 2 \cap N$ such that $b\rst u_N=f$.
Fix $p\in \Pcal_N$ which has no extensions in $E_f$.
Let
\[
W =  \{ K \in [ H_\theta ]^{\omega_1} : p \; \widehat \;  \langle J,K \rangle  \in \Pcal, \mbox{ for some } J \},
\]
\noindent and let $U  =  \{ K \cap \lambda : K \in W \}$.
Note that $W , U \in N$ and
that $W$ and $U$ are $\subseteq$-cofinal in $[ H_\theta ]^{\omega_1}$ and
$[ \lambda ]^{\omega_1}$, respectively.
Note also that, by the choice of $p$, $f\rst u\in A(u)$, for all $u \in U \cap N$.
Let $\G$ be the set of all $g \in \F$ such that
\begin{enumerate}
\item[(i)] $g \rst u \in A(u)$, for all $u \in U$ with $u \subseteq \dom (g)$,
\item[(ii)]  for any $u \in U$ with $\dom (g) \subseteq u$ there exists
$h\in A(u)$ with $g \subseteq h$.
\end{enumerate}
Since all the parameters in the definition of $\G$ are in $N$ and $N$
is elementary in $H_\mu$, it follows that $\G \in N$.
Note also that $f \rst v \in \G$ for all $v \in [ \lambda ]^{\leq \omega_1} \cap N$
by the fact that $f \rst u \in A_u$ for all $u \in U \cap N$ and the elementarity of $N$.
It follows that $\mrm{lev}_u ( \G )$ is nonempty, for all $u \in [ \lambda ]^{\leq \omega_1}$,
again by the elementarity of $N$. Clearly $\G$ is closed under restrictions.
So $\G$ is an $( \omega_2 , \lambda )$-tree.
Moreover, all the levels of  $\mrm{lev}_u ( \G )$ are countable.
Let $\B \subseteq {}^\lambda 2$ and $C \subseteq [ \lambda ]^{\leq \omega_1}$
be those obtained by applying Lemma \ref{lem:very_thin_tree} for $\G$.
We may assume that $\B , C \in N$ by the elementarity of $N$.
Take an $\subseteq$-increasing sequence $(u_n)_n$  of elements of $C\cap N$
such that  that $\bigcup \{ u_n : n<\omega \}  = u_N$.
Moreover, for each $n$, let $b_n$ be the unique element of $\B$
with $b_n \rst u_n  = f \rst u_n $.
Note that $b_m = b_n$ for each $m , n$ by the uniqueness.
Therefore $b_0 \rst u_N = f$.
Moreover $b_0 \in N$, since $\B \in N$ and $\B$ is countable. Therefore $b_0$ is as desired.
\end{proof}

Now, let $\Dcal$ be the set of the $D_K$, for $K \in [ H_\theta ]^{\omega_1} \cap N$,
and $\Ecal$ the set of the $E_f$, for $f \in \mrm{lev}_{u_N} ( \F )$ such that
there is no $b\in {}^\lambda 2 \cap N$ with $b\rst u_N=f$.
Then $\Dcal$ and $\Ecal$ are dense subsets of $\Pcal_N$. Moreover, $\Dcal$ is countable
and $\Ecal$ has cardinality at most $\aleph_1$, since the cardinality of $ \mrm{lev}_{u_N} ( \F )$
is at most $\aleph_1$.
By $\mrm{MA}_{\omega_1} ( \mrm{Cohen} )$ we can find a filter $G$ in $\Pcal_N$ which meets
all the sets of $\Dcal \cup \Ecal$.
Let $r_G = \bigcup G$. Then $r_G$ is an infinite run of the game $G_2(\theta,F,\xi)$ in which II follows $\tau$
and therefore wins. Let us say $r_G=\langle J_0,K_0,J_1,K_1,\ldots \rangle$ and
let $u_n=K_n\cap \lambda$, for all $n$.  Let
\[
M = \clsr_F ( \xi \cup \{ K_n : n  \in \omega \} ).
\]
We show that this $M$ is as desired.
Prior to this, note that
$M$ is an elementary submodel of  $(H_\theta , \in , \unlhd)$ and contains
$\lambda$ and $\F$. Moreover,  $M \cap \omega_1 = \xi$, since II wins the play $r_G$.
Moreover, $u_M = u_N$ since $G$ meets all the dense sets in $\Dcal$.

Now,  $f \in \mrm{lev}_{u_M} ( \F )$ and suppose first that (1) fails for $f$,
i.e. there is no $b\in {}^\lambda 2\cap M$ such that $b\restriction u_M=f$.
Then $G \cap E_f \neq \emptyset$. Let $n$ be such that
$f \rst  u_n  \notin A(u_n)$.
Note that $A(u_n) = \mrm{lev}_{u_n} ( \F ) \cap M$,
since $u_n=K_n \cap \lambda \in M$  and $M \cap \omega_1 = \xi$.
So $f \rst u_n \notin M$.
Thus $u_n$ witnesses (2) for $f$.
Next suppose that there exists $b \in {}^\lambda 2 \cap N$ such that $b \rst u_M = f$.
Then we can find an integer $n$ such that $b \in K_n$, by the $\Dcal$-genericity of $G$.
Since $K_n \in M \cap [ H_\theta ]^{\omega_1}$,
we can find $u \in [ \lambda ]^{\leq \omega_1} \cap M$ such that
$c \rst u \neq d \rst u$ for any distinct $c , d \in {}^\lambda 2 \cap K_n$.
Then $c \mapsto c \rst u$ is an injection from ${}^\lambda 2 \cap K_n$ to
${}^u 2$, and this injection belongs to $M$.
Hence $c \in M$ if and only if $c \rst u \in M$, for any
$c \in {}^\lambda 2 \cap K_n$.
Here note that $b \notin M$ by our assumption that (1) fails for $f$.
Therefore $f \rst u = b \rst u \notin M$.
\end{proof}

Using Lemma \ref{lem:ma_tree} we prove Theorem \ref{thm:itp_tp}.

\begin{proof}[Proof of Theorem \ref{thm:itp_tp}]
(a) Assume $\sref$ and $\mrm{MA}_{\omega_1} ( \mrm{Cohen} )$.
Let $\lambda$ be an ordinal $\geq \omega_2$, $\F$ be a thin $( \omega_2 , \lambda )$-tree
and $\vec{f} = ( f_u : u \in [ \lambda ]^{\leq \omega_1} )$
be a level sequence of $\F$.
We will find an ineffable branch for $\vec{f}$.
Let $\theta$ be a sufficiently large regular cardinal,
let $\unlhd$ be a well-ordering of $H_\theta$,
and let $\frak A= ( H_\theta , \in , \unlhd)$.
Moreover let $Z$ be the set of all countable elementary submodels $M$
of $\frak A$ containing $\lambda$ and $\F$
and such that either (1) or (2) of Lemma \ref{lem:ma_tree} holds,
for any $f \in \mrm{lev}_{u_M} ( \F )$.
Then $Z$ is stationary in $[ H_\theta ]^\omega$ by Lemma \ref{lem:ma_tree}.
Let $W$ be the set of all $K \in [ H_\theta ]^{\omega_1}$ such that
$\omega_1 \subseteq K$ and  $Z \cap [K]^\omega$ is stationary in $[K]^\omega$.
By $\sref$ it follows that $W$ is stationary in $[ H_\theta ]^{\leq \omega_1}$.

\begin{claim*}
For any $K \in W$ there is $b_K \in {}^\lambda 2 \cap K$
such that $b_K \rst ( K \cap \lambda ) = f_{K \cap \lambda}$.
\end{claim*}

\noindent
\begin{proof} Fix $K \in W$, and let $f = f_{K \cap \lambda}$.
Note that $K$ is an elementary submodel of  $\frak A$ since $Z \cap [K]^\omega$ is stationary.
Then $\mrm{lev}_u ( \F )$ is a  subset of $K$, for all $u \in [ \lambda ]^{\leq \omega_1} \cap K$,
since $\mrm{lev}_u ( \F )$ it is an element of $K$ of size $\aleph_1$ and $\omega_1\subseteq K$.
Since $Z \cap [K]^\omega$ is stationary in $[K]^\omega$, we can find $M \in Z \cap [K]^\omega$,
such that, letting $N=\skull^{\frak A} ( M \cup \{ f \} )$, we have that $N\cap K =M$.
Note that $f \rst u \in M$, for all $u \in [ \lambda ]^{\leq \omega_1} \cap M$,
since, $f\rst u \in K\cap N$, for every such $u$.
So (2) of Lemma \ref{lem:ma_tree} fails for $f \rst u_M$ and $M$.
Hence (1) holds for $f \rst u_M$ and $M$,
that is, there is $b \in {}^\lambda 2 \cap M$ with $b \rst u_M = f \rst u_M$.
Note that $b \in K$. Hence it suffices to show that $b \rst ( K \cap \lambda ) = f$.
Note that both $b \rst ( K \cap \lambda )$ and $f$ are functions on $K \cap \lambda$
which are in $N$.
Moreover $b \rst ( K \cap \lambda )$ and $f$ coincides on
$N \cap ( K \cap \lambda )$ since $N \cap ( K \cap \lambda ) = M \cap \lambda$,
and $b \rst u_M = f \rst u_M$.
Hence $b \rst ( K \cap \lambda ) = f$, by the elementarity of $N$.
\end{proof}

\medskip

By the Pressing Down Lemma we can find $b \in {}^\lambda 2$ such that
$b_K=b$, for stationary many $K \in W$.
It follows that $b$ is an ineffable branch for $\vec{f}$.

\bigskip

\noindent
(b) Assume $\ssref$ and $\mrm{MA}_{\omega_1} ( \mbox{Cohen} )$.
Let $\lambda$ be an ordinal $\geq \omega_2$
and $\F$ be a thin $( \omega_2 , \lambda )$-tree.
We will find a cofinal branch for $\F$.
Let $\theta$, $\unlhd$, $\frak A$ and $Z$ be as in the proof of (a).
Moreover let $Z^*$ be the upward closure of $Z$ under $\sqsubseteq$.
By $\ssref$ there is $K \in [ H_\theta ]^{\omega_1}$
such that $\omega_1 \subseteq K$ and $Z^* \cap [K]^\omega$ is stationary in $[K]^\omega$.
Here note that $Z^* \cap [K^*]^\omega$ is stationary, for any $K^* \supseteq K$.
Hence, by replacing $K$ with $\skull^{\frak A} (K)$ if necessary,
we may assume that $K$ is an elementary submodel of  $\frak A$ and
contains $\lambda$ and $\F$ as elements.
Pick any $f \in \mrm{lev}_{K \cap \lambda} ( \F )$.
Then $f \rst u \in K$, for all $u \in [ \lambda ]^{\leq \omega_1} \cap K$.
So we can take $M^*\in Z^* \cap [K]^\omega$ which is an elementary submodel of
$\frak A$, contains $\lambda$ and $\F$ as elements,  and such that $f \rst u \in M^*$,
for all $u \in [ \lambda ]^{\leq \omega_1} \cap M^*$.
Let $M \in Z$ be such that $M \sqsubseteq M^*$.
Here note that $\mrm{lev}_u ( \F ) \cap M = \mrm{lev}_u ( \F ) \cap M^*$,
for any $u \in [ \lambda ]^{\leq \omega_1} \cap M$, since both
$M$ and $M^*$ are elementary submodels of $\frak A$,
$\mrm{lev}_u ( \F )$ is of size $\aleph_1$, and  $M \cap \omega_1 = M^* \cap \omega_1$.
Hence $f \rst u \in M$, for all $u \in [ \lambda ]^{\leq \omega_1} \cap M$,
and so (2) of Lemma \ref{lem:ma_tree} fails for $f \rst u_M$ and $M$.
Thus there is $b \in {}^\lambda 2 \cap M$ with $b \rst u_M = f \rst u_M$.
Then $b \rst u \in \mrm{lev}_u ( \F )$, for all $u \in [ \lambda ]^{\leq \omega_1} \cap M$.
So $b \rst u \in \mrm{lev}_u ( \F )$, for all $u \in [ \lambda ]^{\leq \omega_1}$,
by the elementarity of $M$. So, $b$ is a cofinal branch of $\F$, as required.
\end{proof}

We now show that $\ssref$ and $\mrm{MA}_{\omega_1} ( \mrm{Cohen} )$
is not sufficient to imply the super tree property for $\omega_2$.

\begin{theorem} \label{thm:ssr_itp}
Assume that there exists a strongly compact cardinal.
Then there exists a forcing extension in which $\ssref$ and $\mrm{MA}_{\omega_1} ( \mrm{Cohen} )$
hold but $\omega_2$ does not have the super tree property.
\end{theorem}

Theorem \ref{thm:ssr_itp} follows easily from the following facts.

\begin{fact}[Magidor] \label{fact:identity_crisis}
Assume that $\kappa$ is a supercompact cardinal.
Then there is a forcing extension in which $\kappa$ is strongly compact
but not supercompact.
\end{fact}

\begin{fact}[Shelah {\cite[Chapter XIII, 1.6 and 1.10]{Sh_P}}] \label{fact:strongly_compact_ssr}
Assume that $\kappa$ is a strongly compact cardinal.
Let $( \Pcal_\alpha , \dot{\Qcal}_\beta : \alpha \leq \kappa , \beta < \kappa )$
be a revised countable support iteration of semi-proper posets  of size $< \kappa$
such that $\kappa = \omega_2$ in $V^{\Pcal_\kappa}$.
Then $\ssref$ holds in $V^{\Pcal_\kappa}$.
\end{fact}

\begin{fact}[Viale-Weiss \cite{VW}] \label{fact:itp_supercompact}
Assume that $\kappa$ is an inaccessible cardinal.
Assume also that there exists a countable support iteration
$( \Pcal_\alpha , \dot{\Qcal}_\beta : \alpha \leq \kappa , \beta < \kappa )$
of proper posets of size $< \kappa$
such that $\kappa$ has the super tree property in $V^{\Pcal_\kappa}$.
Then $\kappa$ is supercompact in $V$.
\end{fact}

\begin{proof}[Proof of Theorem \ref{thm:ssr_itp}]
Assume that $\kappa$ is strongly compact in $V$.
By Fact \ref{fact:identity_crisis} we may assume that $\kappa$ is not supercompact.
Let $( \Pcal_\alpha , \dot{\Qcal}_\beta : \alpha \leq \kappa , \beta < \kappa )$
be the countable support iteration of  Cohen forcing.
Here recall that a revised countable support iteration coincides with
a countable support iteration for proper posets.
Note also that $\kappa = \omega_2$ in $V^{\mbb{P}_\kappa}$.
Hence $\ssref$ holds $V^{{\Pcal}_\kappa}$ by Fact \ref{fact:strongly_compact_ssr}.
Moreover $\mrm{MA}_{\omega_1} ( \mbox{Cohen} )$ holds in $V^{{\Pcal}_\kappa}$.
By Fact \ref{fact:itp_supercompact}, $\omega_2$ does not have the super tree property in
$V^{{\Pcal}_\kappa}$.
\end{proof}

We end this section with some remarks. In Theorem \ref{thm:ssr_itp} we have proved that
$\ssref + \mrm{MA}_{\omega_1} ( \mrm{Cohen} )$ does not imply the super tree property
at $\omega_2$. In fact we can prove that it does not imply $\mrm{ITP} ( \omega_2 , \omega_3 )$.
We outline the proof.  For a regular uncountable cardinal $\kappa$ let
\[
U_\kappa = \{ u \in [ \kappa^+ ]^{< \kappa} :
u \cap \kappa \in \kappa \,\wedge\, \otype (u) = ( u \cap \kappa )^+ \} \; .
\]
It is easy to see that if $\kappa$ is $\kappa^+$-supercompact,
then $U_\kappa$ is stationary in $[ \kappa^+ ]^{< \kappa}$.
On the other hand Krueger \cite{Kr} proved that
this does not follow from the strong compactness of $\kappa$.

\begin{fact}[Krueger \cite{Kr}] \label{fact:strongly_compact_U}
Assume that $\kappa$ is a supercompact cardinal.
Then there is a forcing extension in which $\kappa$ is strongly compact,
and $U_\kappa$ is non-stationary.
\end{fact}

\noindent
Moreover we can prove the following.

\begin{fact} \label{fact:itp_U}
Assume that $\kappa$ is an inaccessible cardinal.
Assume also that there is a countable support iteration
$( \Pcal_\alpha , \dot{\Qcal}_\beta : \alpha \leq \kappa , \beta < \kappa )$
of proper posets of size $< \kappa$ such that
$\mrm{ITP} ( \kappa , \kappa^+ )$ holds in $V^{\Pcal_\kappa}$.
Then $U_\kappa$ is stationary in $[ \kappa^+ ]^{< \kappa}$ in $V$.
\end{fact}

\noindent
Using Facts \ref{fact:strongly_compact_U} and \ref{fact:itp_U}
instead of Facts \ref{fact:identity_crisis} and \ref{fact:itp_supercompact},
by the same argument as Theorem \ref{thm:ssr_itp},
we can prove that $\ssref + \mrm{MA}_{\omega_1} ( \mrm{Cohen} )$
does not imply $\mrm{ITP} ( \omega_2 , \omega_3 )$.

On the other hand, we can also prove that $\ssref + \mrm{MA}_{\omega_1} ( \mrm{Cohen} )$
implies $\mrm{ITP} ( \omega_2 , \omega_2 )$:
Assume $\ssref$, and suppose that $\F$ is a thin $( \omega_2 , \omega_2 )$-tree
and that $\vec{f} = ( f_u : u \in [ \omega_2 ]^{\leq \omega_1} )$
is a level sequence of $\F$.
Let $\theta$, $\unlhd$, $\frak A$ and $Z$ be as in the proof of Theorem \ref{thm:itp_tp}.
Here recall the fact, due to Foreman, Magidor and Shelah \cite{FMS},
that $\ssref$ (equivalently $( \dagger )$) implies Strong Chang's Conjecture.
In fact it implies the following.
\begin{quote}
There are club many $M \in [ H_\theta ]^\omega$ such that
$\skull^{\frak A} ( M \cup \{ \delta \} ) \cap \delta = M \cap \omega_2$,
for stationary many $\delta \in \omega_2$.
\end{quote}
Take such $M \in Z$, and let $E$ be the set of all $\delta \in \omega_2$
with $\skull^{\frak A} ( M \cup \{ \delta \} ) \cap \delta = M \cap \omega_2$.
Then, by the same argument as in the proof of Theorem \ref{thm:itp_tp} (a),
for any $\delta \in E$ there is $b_\delta \in {}^{\omega_2} 2 \cap M$
such that $b_\delta \rst \delta = f_\delta$.
Take $b \in {}^{\omega_2} 2$ such that
$\{ \delta \in E : b_\delta = b \}$ is stationary.
Then $b$ is an ineffable branch for $\vec{f}$.


\section{Singular cardinal hypothesis} \label{sec:sch}

In this section we prove that $\ssref$ implies $\mrm{SCH}$.
In fact we prove the following:

\begin{theorem} \label{thm:better_scale}
Assume that $\lambda$ is a singular cardinal of cofinality $\omega$
and that $\ssref ( \lambda^+ )$ holds.
Then for any strictly increasing sequence $\vec{\lambda} = ( \lambda_n : n < \omega )$
of regular cardinals converging to $\lambda$ there are no better scales on $\vec{\lambda}$.
\end{theorem}

\noindent
Here recall the fact, due to Foreman, Magidor and Shelah \cite{FMS},
that $\ssref$ implies the Strong Chang's Conjecture
and the fact, due to Todor\v{c}evi\'{c} \cite{Todorcevic_cc},
that the Strong Chang's Conjecture implies $2^\omega \leq \omega_2$.
So $\ssref$ implies $2^\omega \leq \omega_2$.
Then, using the theorem above and Fact \ref{fact:better_scale},
it is easy to see that if $\ssref$ holds, then $\lambda^\omega = \lambda^+$
for all singular cardinals of cofinality $\omega$.
Then it follows from Fact \ref{fact:Silver} that if $\ssref$ holds,
then so does $\mrm{SCH}$.

\begin{cor} \label{cor:sch}
$\ssref$ implies $\rm SCH$.
\qed
\end{cor}

To prove the theorem we make preliminaries.
First we present a game which is a variant of the game used in Section \ref{sec:wsquare}:

Let $\vec{\lambda} = ( \lambda_n : n \in \omega )$ be a strictly increasing
sequence of regular cardinals $\geq \omega_2$,
let $\lambda = \sup_{n \in \omega} \lambda_n$,
and let $\vec{E} = ( E_{n,i} : n \in \omega , \; i \in 2 )$ be a sequence such that
each $E_{n,i}$ is a stationary subset of $E^{\lambda_n}_{\omega_1}$.
Moreover let $\vec{b} = ( b_\xi : \xi < \omega_1 )$
be a sequence of functions from $\omega$ to $2$.
For a function $F : [ \lambda^+ ]^{< \omega} \to \lambda^+$ and $\xi < \omega_1$
let $G_3 ( \vec{E} , \vec{b} , F , \xi )$ be the following game of length $\omega$:
\[
\begin{array}{c |c|c|c|c|c}
\mrm{I} & \alpha_0 \ \phantom{\beta_0 , \gamma_0} \ \delta_0 , \epsilon_0 &
\alpha_1 \ \phantom{\beta_1 , \gamma_1} \ \delta_1 , \epsilon_1 & \cdots &
\alpha_n \ \phantom{\beta_n , \gamma_n} \ \delta_n , \epsilon_n & \cdots \\
\hline
\mrm{II} & \phantom{\alpha_0} \ \beta_0 , \gamma_0 \ \phantom{\delta_0 , \epsilon_0} &
\phantom{\alpha_1} \ \beta_1 , \gamma_1 \ \phantom{\delta_1 , \epsilon_1} & \cdots &
\phantom{\alpha_n} \ \beta_n , \gamma_n \ \phantom{\delta_n , \epsilon_n} & \cdots \\
\end{array}
\]
In the $n$-th stage, first I chooses $\alpha_n < \lambda_n$,
then II chooses $\beta_n < \lambda_n$ and $\gamma_n < \lambda^+$.
Then I again chooses $\delta_n > \alpha_n , \beta_n$ with $\delta_n \in E_{n , b_\xi (n)}$
and $\epsilon_n < \lambda^+$ with $\epsilon_n > \gamma_n$.
I wins if, letting $x=\clsr_F ( \xi \cup \{ \delta_n , \epsilon_n : n \in \omega \} )$,
we have:
\begin{renumerate}
\item $x\cap \omega_1 = \xi$,
\item $x \cap [ \alpha_m , \delta_m ) = \emptyset$, for every $m$.
\end{renumerate}
Otherwise, II wins.

\begin{lemma} \label{lem:game3}
Let $\vec{\lambda}$, $\lambda$, $\vec{E}$, $\vec{b}$ and $F$ be as above.
Then there exists $\xi < \omega_1$ such that I has a winning strategy
for $G_3 ( \vec{E} , \vec{b} , F , \xi )$.
\end{lemma}

\begin{proof}
Towards a  contradiction assume that I does not have a winning strategy
in $G_3 ( \vec{E} , \vec{b} , F , \xi )$, for every $\xi < \omega_1$.
Since each $G_3 ( \vec{E} , \vec{b} , F , \xi )$ is an open-closed game,
by the Gale-Stewart theorem, II has a winning strategy, say $\tau_\xi$,
for all $\xi$. We will find $\zeta < \omega_1$ and a play
$( \alpha_n , \beta_n , \gamma_n , \delta_n , \epsilon_n : n \in \omega )$
of $G_3 ( \vec{E} , \vec{b} , F , \zeta )$ in which II follows his strategy
$\tau_\zeta$, yet I wins the game. 

Let $\vec{\tau} = ( \tau_\xi : \xi < \omega_1 )$.
Take a sufficiently large regular cardinal $\theta$.
Then we can find a system
$( K_{n,i} : n \in \omega , \; i \in 2 )$ of elementary submodels
of $( H_\theta , \in)$ containing  $\vec{\lambda}$ and $\vec{\tau} )$,
such that $\delta_{n,i} = K_{n,i} \cap \lambda_n \in E_{n, i}$, for each $n$ and $i$,
and such that $K_{n,i} \in K_{n' , i'}$ if $n < n'$ and $i, i' \in 2$.
Let $\epsilon_{n,i} = \bsup ( K_{n,i} \cap \lambda^+ )$, for each $n$ and $i$.
Then we can take $\zeta < \omega_1$ such that
\[
\clsr_F ( \zeta \cup \{ \delta_{n,i} , \epsilon_{n,i} : n \in \omega , \; i \in 2 \} ) \cap \omega_1
= \zeta.
\]
For each $n$, let $K_n$, $\delta_n$ and $\epsilon_n$
be $K_{n , b_\zeta (n)}$, $\delta_{n , b_\zeta (n)}$ and $\epsilon_{n , b_\zeta (n)}$,
respectively. Moreover, let
\[
x = \clsr_F ( \zeta \cup \{ \delta_n , \epsilon_n : n \in \omega \} ) \; ,
\]
and let $\alpha_n = \bsup ( x \cap \delta_n )$ for each $n$.
Note that $\alpha_n < \delta_n$, since $x$ is countable and $\cof ( \delta_n ) = \omega_1$.
Finally let $( \beta_n , \gamma_n : n \in \omega )$ be a sequence of
II's moves according to $\tau_\zeta$ against
$( \alpha_n , \delta_n , \epsilon_n : n \in \omega )$.
Note that
$\zeta, \alpha_0 , \delta_0 , \epsilon_0 , \dots , \alpha_{n-1} , \delta_{n-1} , \epsilon_{n-1} ,
\alpha_n \in K_n$, and $K_n$ is an elementary submodel of $H_\theta$, for all $n$. 
Hence $\beta_n \in  K_n \cap \lambda_n = \delta_n$.
Moreover $\gamma_n \in K_n \cap \lambda^+$,
and so $\gamma_n < \epsilon_n$.
Thus $( \alpha_n , \beta_n , \gamma_n , \delta_n , \epsilon_n : n \in \omega )$
is a legal play of $G_3 ( \vec{E} , \vec{b}, F , \zeta )$
in which II moves according to his winning strategy $\tau_\zeta$.
On the other hand $x \cap [ \alpha_m , \delta_m ) = \emptyset$,
for every $m$, by the choice of $\alpha_m$ and $\delta_m$.
Moreover $x \cap \omega_1 = \zeta$ by the choice of $\zeta$ and $x$.
Therefore I wins this play of the game. This is a contradiction.
\end{proof}

Next we give a standard lemma on better scales.

\begin{lemma} \label{lem:better_scale_stat}
Let $\vec{\lambda} = ( \lambda_n : n < \omega )$
be a strictly increasing sequence of regular cardinals,
and let $\lambda = \sup_{n \in \omega} \lambda_n$.
Suppose that $\vec{f} = ( f_\beta : \beta < \lambda^+ )$
is a better scale on $\vec{\lambda}$.
Then for any regular $\theta > \lambda^+$
there are stationary many $N \in [ H_\theta ]^\omega$
with $\chi_N \leq^* f_{\bsup ( N \cap \lambda^+ )}$,
where $\chi_N$ is the characteristic function of $N$ (see \S \ref{sec:preliminaries}).
\end{lemma}

\begin{proof}
Suppose that $\theta$ is a regular cardinal bigger than $\lambda^+$.
It is sufficient to show that for every expansion $\frak A$ of $(H_\theta,\in)$
there is a countable elementary submodel $N$ of $\frak A$, such that 
$\chi_N \leq^* f_\rho$, where $\rho = \sup ( N \cap \lambda^+ )$.
In order to find such an $N$, first take a continuous 
$\in$-chain $( N_\xi : \xi < \omega_1 )$ of countable elementary submodels of $\frak A$
containing all the relevant parameters. 
Let $\rho_\xi = \sup ( N_\xi \cap \lambda^+ )$.
Then, since $\vec{f}$ is a better scale, we can find $m < \omega$ and
a stationary $S \subseteq \omega_1$ such that
$( f_{\rho_\xi} : \xi \in S )$ is $<_m$-increasing.
Here note that if $\xi < \eta$, then $\chi_{N_\xi} <^* f_{\rho_{\eta}}$,
since $N_\xi \in N_{\eta}$ and $N_\eta$ is an elementary submodel of $\frak A$.
So, by shrinking $S$ and increasing $m$ if necessary,
we may assume that $\chi_{N_\xi} <_m f_{\rho_{\eta}}$ for any $\xi , \eta \in S$
with $\xi < \eta$. Take $\eta \in \ltpt (S)$.
Then $\chi_{N_\xi} <_m f_{\rho_\eta}$, for all $\xi \in S \cap \eta$.
Moreover $\chi_{N_\eta} (n) = \sup_{\xi \in S \cap \eta} \chi_{N_\xi} (n)$, for all $n$,
since $N_\eta = \bigcup_{\xi \in S\cap \eta} N_\xi$.
So $\chi_{N_\eta} \leq_m f_{\rho_\eta}$.
Therefore $N= N_\eta$ is as desired.
\end{proof}

Now we prove Theorem \ref{thm:better_scale}.
In the proof we will use Lemma \ref{lem:ssr_sup} as well as
Lemmas \ref{lem:game3} and \ref{lem:better_scale_stat}.

\begin{proof}[Proof of Theorem \ref{thm:better_scale}]
Towards a contradiction assume that $\vec{\lambda} = ( \lambda_n : n < \omega )$
is a strictly increasing sequence of regular cardinals converging to $\lambda$
and there is a better scale $\vec{f} = ( f_\beta : \beta < \lambda^+ )$ 
on $\vec{\lambda}$. We may also assume that $\lambda_0 \geq \omega_2$.
Fix a sequence $\vec{E} = ( E_{n,i} : n \in \omega , \; i \in 2 )$
such that $E_{n,0}$ and $E_{n,1}$ are disjoint stationary subsets of 
$E^{\lambda_n}_{\omega_1}$, for all $n$, 
and fix a sequence $\vec{b} = ( b_\xi : \xi < \omega_1 )$ of
functions from $\omega$ to $2$ such that
if $\xi \neq \eta$, then $b_\xi \neq^* b_{\eta}$.
Moreover for each $x \subseteq \lambda^+$
let $e_x$ be the function on $\omega$
defined by $e_x (n) = \min ( x \setminus f_{\bsup (x)} (n) )$.
(If $x \setminus f_{\bsup (x)} (n) = \emptyset$, then let $e_x (n) = 0$.)
Then let $X$ be the set of all $x \in [ \lambda^+ ]^\omega$ such that,
letting $\xi = x \cap \omega_1 \in \omega_1$, we have
\begin{renumerate}
\item $f_{\bsup (x)} < \chi_x$,
\item $e_x (n) \in E_{n , b_\xi (n)}$, for all but finitely many $n$.
\end{renumerate}
Note that $X$ is weakly full. So, it suffices to prove the following two claims. 

\begin{claim}
$X$ is stationary in $[ \lambda^+ ]^\omega$.
\end{claim}
\begin{proof} Take an arbitrary function $F : [ \lambda^+ ]^{< \omega} \to \lambda^+$.
We need to find $x \in X$ which closed under $F$.
By Lemma \ref{lem:game3}, fix  $\xi < \omega_1$ such that
there is a winning strategy $\tau$ of I in $G_3 ( \vec{E} , \vec{b} , F , \xi )$.
Moreover take a sufficiently large regular cardinal $\theta$
and a well-ordering $\unlhd$ of $H_\theta$,
and let $\frak A= ( H_\theta , \in , \unlhd, F,\tau)$.  
Then by Lemma \ref{lem:better_scale_stat} we can find a countable elementary
submodel $N$ of $\frak A$, containing $F$ and $\tau$,  
such that $\chi_N \leq^* f_\rho$, where $\rho = \sup ( N \cap \lambda^+ )$.
Let $\beta_n = f_\rho (n)$, for each $n$,
and take an increasing cofinal sequence $( \gamma_n : n \in \omega )$
in $N \cap \lambda^+$.
Then let $( \alpha_n , \delta_n , \epsilon_n : n \in \omega )$ be a sequence
of I's moves according to $\tau$ against $( \beta_n , \gamma_n : n \in \omega )$,
and let $x = \clsr_F ( \xi \cup \{ \delta_n , \epsilon_n : n \in \omega \} )$.
It suffices to show that $x \in X$.
In order to see this, first note that $x \cap \omega_1 = \xi$, since 
$( \alpha_n , \beta_n , \gamma_n , \delta_n , \epsilon_n : n \in \omega )$
is a play of $G_3 ( \vec{E} , \vec{b} , F , \xi )$ in which I  moves according
to his winning strategy $\tau$.
Note also that $\sup (x) = \rho$. This is because
$\sup (x) \geq \sup_{n \in \omega} \epsilon_n \geq \sup_{n \in \omega} \gamma_n
= \rho$.
On the other hand note that
$x \subseteq \skull^{\frak A} ( \rho ) \cap \lambda^+$. Indeed, $\beta_n,\gamma_n <\rho$,
for all $n$, and $\tau,F\in \skull^{\frak A}(\rho)$. 
Moreover $\skull^{\frak A} ( \rho ) \cap \lambda^+ = \rho$ by Fact \ref{fact:skull_sup}.
Therefore $\sup (x) \leq \rho$.
It follows that $x\cap \omega_1=\xi$. 
Also $x$ satisfies (i),  since
$f_\rho (n) = \beta_n < \delta_n \in x \cap \lambda_n$, for every $n$.
In order to check (ii), first note that $\alpha_n < \chi_N (n)$, for each $n$, since
\[
\alpha_n \in \skull^{\frak A} ( N \cup \chi_N (n) ) \cap \lambda_n = \chi_N (n) \; .
\]
Here the former $\in$-relation is because
$\{ \beta_m , \gamma_m : m < n \} \subseteq N \cup \chi_N (n)$,
and the latter equality is by Fact \ref{fact:skull_sup}.
Then $\alpha_n < f_\rho (n)$, for all but finitely many $n$, since $\chi_N \leq^* f_\rho$.
Note also that $\delta_n \in x$ and that $x \cap [ \alpha_n , \delta_n ) = \emptyset$,
since I wins the play  $( \alpha_n , \beta_n , \gamma_n , \delta_n , \epsilon_n : n \in \omega )$
in $G_3 ( \vec{E} , \vec{b} , F , \xi )$.
Hence $e_x ( n ) = \delta_n$, for all but finitely many $n$.
Moreover, $\delta_n \in E_{n , b_\xi (n) }$ by the rules of $G_3 ( \vec{E} , \vec{b} , F , \xi )$.
Thus $x$ satisfies (ii).
\end{proof}

\bigskip

\begin{claim}
The conclusion of Lemma \ref{lem:ssr_sup} fails for $X$.
\end{claim}
\begin{proof} Towards a contradiction assume that the conclusion of Lemma \ref{lem:ssr_sup} holds for $X$.
Then we can find $u \in [ \lambda ]^{\omega_1}$ such that $\omega_1 \subseteq u$ and
$X \cap [u]^\omega$ is stationary. Clearly $\sup (u)$ is a limit ordinal. We consider
two cases according to whether the cofinality of $\sup (u)$ is $\omega$ or $\omega_1$. 

First suppose that $\cof ( \sup (u) ) = \omega$.
Then the set
\[
Y = \{ x \in [u]^\omega : \sup (x) = \sup (u) \mbox{ and } \mrm{range} ( e_u ) \subseteq x \}
\]
is club in $[u]^\omega$.
Note  that $e_x = e_u$ for all $x \in Y$.
Take $x_0 , x_1 \in X \cap Y$ with $x_0 \cap \omega_1 \neq x_1 \cap \omega_1$,
and let $\xi_i = x_i \cap \omega_1$ for $i = 0,1$.
Then $e_{x_0} = e_u = e_{x_1}$, since $x_0 , x_1 \in Y$,
and so $b_{\xi_0} =^* b_{\xi_1}$, since $x_0 , x_1 \in X$.
This contradicts the choice of $\vec{b}$ and the fact that $\xi_0 \neq \xi_1$.

Next suppose that $\cof ( \sup (u) ) = \omega_1$.
Since $\vec{f}$ is a better scale we can find a club $C$ in $\sup (u)$ and $\sigma : C \to \omega$
such that $f_\beta <_{\max \{ \sigma ( \beta ) , \sigma ( \gamma ) \}} f_\gamma$
for any $\beta , \gamma \in C$ with $\beta < \gamma$.
Let $h$ and $e$ be functions on $\omega$ defined by:
\begin{eqnarray*}
h(n) & = & \sup \{ f_\beta (n) : \beta \in C \,\wedge\, n \geq \sigma ( \beta ) \} \; , \\
e(n) & = & \min ( u \setminus h(n) ) \; .
\end{eqnarray*}
Moreover let $Z$ be the set of all $x \in [u]^\omega$ such that
\begin{renumerate}
\addtocounter{enumi}{3}
\item $\sup (x) \in C$,
\item $x \cap h(n) \subseteq f_{\sup (x)} (n)$, for every $n \geq \sigma ( \sup (x) )$,
\item $\mrm{range} (e) \subseteq x$.
\end{renumerate}
Then it is easy to see that $Z$ contains a club subset of $[u]^\omega$.
Here note that if $x \in Z$, then $e_x (n) = e(n)$ for all $n \geq \sigma ( \bsup (x) )$.
Then we can get a contradiction by the same argument as in the case
when $\cof (\sup (u) ) = \omega$.
\end{proof}
\setcounter{claim}{0}
This completes the proof of Theorem \ref{thm:better_scale}.
\end{proof}

\bibliographystyle{plain}
\bibliography{ssr_final}

\end{document}